\documentclass[english,11pt]{smfart}

  \usepackage{mathrsfs,dsfont,smfthm}
  \usepackage{graphicx}
  \usepackage{color}
  \usepackage{pinlabel}

  \usepackage[utf8]{inputenc}
  \usepackage{amsthm}
  \usepackage{amsmath,amsxtra,amscd,amssymb,latexsym,stmaryrd}
  \usepackage[T1]{fontenc}

\title{Curvatures and anisometry of maps}

\author{Beno\^{\i}t R. Kloeckner}
\address{Universit\'e de Grenoble I, Institut Fourier\\ CNRS UMR 5582\\ BP 74\\
   38402 Saint Martin d'H\`eres cedex\\ France}
\email{benoit.kloeckner@ujf-grenoble.fr}

\newcommand{\Ric}{\operatorname{Ric}}
\newcommand{\Scal}{\operatorname{Scal}}
\newcommand{\inj}{\operatorname{inj}}
\newcommand{\jac}{\operatorname{jac}}
\newcommand{\Vol}{\operatorname{Vol}}
\newcommand{\aniso}{\operatorname{aniso}}
\newcommand{\dist}{\operatorname{dist}}
\newcommand{\Candle}{\operatorname{Candle}}
\newcommand{\LCD}{\operatorname{LCD}}
\renewcommand{\le}{\leqslant}
\renewcommand{\ge}{\geqslant}
\newcommand{\dd}{\mathrm{d}}
\newcommand{\argth}{\operatorname{argth}}

\theoremstyle{definition}
\newtheorem*{defi*}{Definition}
\newtheorem*{rema*}{Remark}

\begin{document}
%%%%%%%%%%%%%%%%%%%%%%%%%%%%%%%%%%%%%%%%%%%%%%%%%%%%%%%%%%%%%%%
%%%%%%%%%%%%%%%%%%%%%%%%%%%%%%%%%%%%%%%%%%%%%%%%%%%%%%%%%%%%%%%
%%%%%%%%%%%%%%%%%%%%%%%%%%%%%%%%%%%%%%%%%%%%%%%%%%%%%%%%%%%%%%%

\begin{abstract}
We prove various inequalities measuring how far from an isometry
a local map from a manifold of high curvature to a manifold of low curvature
must be. We consider the cases of volume-preserving, conformal and
quasi-conformal maps. The proofs relate to a conjectural
isoperimetric inequality for manifolds whose curvature is bounded above, and
to a higher-dimensional generalization of the Schwarz-Ahlfors lemma.
\end{abstract}

\maketitle

%%%%%%%%%%%%%%%%%%%%%%%%%%%%%%%%%%%%%%%%%%%%%%%%%%%%%%%%%%%%%%%
%%%%%%%%%%%%%%%%%%%%%%%%%%%%%%%%%%%%%%%%%%%%%%%%%%%%%%%%%%%%%%%
\section{Introduction}

One of the basic facts of Riemannian geometry is that curvatures
are isometry invariants: this explains for example why one cannot
design a perfect map of a region on the earth. In this article, we
shall be interested in quantifying this fact: how far from being an isometry
a map from a region of a manifold to another manifold must be,
when the source and target manifolds satisfy incompatible curvature
bounds?

When the source manifold is the round $2$-sphere and the target manifold
is the Euclidean plane, this question is a cartography problem:
a round sphere is a relatively good approximation of the shape
of the Earth. It has been considered by Milnor \cite{Milnor} who
described the best map when the source region is a spherical cap.
Surprisingly, it seems like no other cases of the general question above
have been considered.

%%%%%%
\subsection{Distortion and anisometry}

To fill this gap, one has first to ask how we should measure the isometric default
of a map $\varphi:D\subset M \to N$ from a domain in a manifold $M$ to a manifold $N$,
assumed to be a diffeomorphism on its image. Milnor uses the distortion,
defined as follows. Let $\sigma_1=\sigma_1(\varphi)$ and $\sigma_2=\sigma_2(\varphi)$
be the Lipschitz constants of $\varphi$, i.e.
\[\sigma_1 d(x,y) \le d(\varphi(x),\varphi(y)) \le \sigma_2 d(x,y) \quad \forall x,y\in D\]
and $\sigma_1$, $\sigma_2$ are respectively the greatest and least numbers satisfying 
such an inequality. Then the \emph{distortion} of $\varphi$ is the number
$\dist(\varphi) = \log(\sigma_2/\sigma_1)$.

However, when the target manifold is not Euclidean, the distortion is ill-suited: it is
zero for maps that are not isometries, but mere homotheties. More disturbing is the case
when $M$ is positively curved and $N$ is negatively curved: to minimize distortion, one
is inclined to take $\varphi$ with a very small image, so that the curvature of $N$ barely
matters. To make this case more interesting, we propose the following definition of 
\emph{anisometry}:
\[\aniso(\varphi) = |\log \sigma_1| + |\log \sigma_2|.\]
This quantity generalizes distortion in the sense that when $N=\mathbb{R}^n$,
\[\inf_\varphi \aniso(\varphi) = \inf_\varphi \dist(\varphi).\]

%%%%%%
\subsection{Azimuthal maps}

To describe our results we will need to introduce a specific family of maps
between model spaces. All considered manifolds will be of the same fixed dimension $n$;
we set $X_\kappa$ for the simply connected manifold of constant curvature $\kappa$
(thus a sphere, the Euclidean space or a hyperbolic space).

Given a point $x\in X_\kappa$, we have polar coordinates $(t,u)$ ($t$ a positive real,
$u$ a unit tangent vector at $x$) given by the exponential map:
\[y = \exp_x(tu)\]
where $t$ is less than the conjugate radius and $y$ may be any point but the antipodal
point to $x$ (when $\kappa>0$).

\begin{defi*}
An \emph{azimuthal} map is a map $\varphi: D\subset X_\rho \to X_\kappa$
where $D$ is a geodesic ball, which reads in polar coordinates centered
at $x$ and $\varphi(x)$ as
\[\varphi(t,u)=(R(t),L(u))\]
where $L$ is a linear isometry from $T_x X_\rho$ to $T_{\varphi(x)} X_\kappa$
and $R$ is a differentiable function. In other words, we have
\[\varphi(\exp_x(tu)) =\exp_{\varphi(x)}(R(t) L(u)).\]

The function $R$ is then called the \emph{distance function} of $\varphi$.
\end{defi*}

As we consider only model spaces, $L$ is irrelevant and the function $R$ defines 
a unique azimuthal map up to isometries.
The azimuthal map associated to each of the following distance functions bears
a special name:
\begin{itemize}
\item $R(t)=t$: \emph{equidistant} azimuthal map,
\item $R(t)= \sigma t$ with $\sigma\in (0,1)$: $\sigma$-\emph{contracting} azimuthal map
\end{itemize}
Moreover, given $\rho$ and $\kappa$ there exists exactly one family of conformal azimuthal
maps and a unique volume-preserving map $B_\rho(\alpha)\to X\kappa$ (see below for details).

%%%%%
\subsection{Description of the results}

We shall not state our results in the greatest generality in this introduction,
please see below for details.

Our main results have the following form: we assume $M$ satisfies some
kind of lower curvature bound associated with a parameter $\rho$, 
that $N$ satisfies some kind of upper curvature bound (or more general 
geometric assumption) associated with a parameter $\kappa<\rho$, and
that $\varphi$ is a map (possibly satisfying extra assumptions)
from a geodesic ball of center $x$ and radius $\alpha$
in $M$ to $N$. 

Our methods provide half-local results, and we
 shall always assume that $\alpha$ is bounded above by 
some number. This bound shall be explicit most of the time and depends only on synthetic
geometrical properties of $M$ and $N$. In some cases (e.g. when
the target is a Hadamard manifold) this bound will be completely harmless.

We then conclude that there is an azimuthal map 
$\bar\varphi:B_\rho(\alpha) \to X_\kappa$ (where $B_\rho(\alpha)$
is any geodesic closed ball of radius $\alpha$ in $X_\rho$) such that
\[\aniso(\varphi) \ge \aniso(\bar \varphi)\]
with equality if and only if $\varphi$ is conjugated
to $\bar \varphi$ by isometries.

For simplicity, we shall write $\Ric_M\ge\rho$ to mean
that the Ricci tensor and the metric tensor of $M$ satisfy
the usual bound
\[\Ric_x(u,u) \ge \rho\cdot (n-1) g_x(u,u) \quad \forall x,u.\]
Similarly, $K_N\le \kappa$ means that the sectional 
curvature of $N$ is not greater than $\kappa$ at any
tangent $2$-plane.

We shall always assume implicitly that $M$ (or more generally
$B_M(x,\alpha)$) and $N$ are complete; recall that they have the
same dimension $n$.

\begin{theo}[General maps]\label{theo:intro-general}
Assume $\Ric_M\ge\rho$, $K_N\le \kappa$ where $\rho >\kappa$,
and $\alpha \le A_1(M,N)$ where $A_1(M,N)$ is an explicit positive
constant.

Then any map
$\varphi:B(x,\alpha)\subset M \to N$ satisfies
\[\aniso(\varphi) \ge \aniso(\bar\varphi)\]
where $\bar \varphi$ is:
\begin{itemize}
\item the equidistant azimuthal map $B_\rho(\alpha)\to X_\kappa$ when $\kappa\ge 0$,
\item the $\sigma$-contracting azimuthal map $B_\rho(\alpha)\to X_\kappa$ when
  $\kappa<0$, where $\sigma$ is such that the boundaries of $B_\rho(\alpha)$
  and $B_\kappa(\sigma\alpha)$ have equal volume.
\end{itemize}
Moreover in case of equality $\varphi$ and $\bar\varphi$ are conjugated
by isometries (in particular, the source and image of $\varphi$ have constant
curvature $\rho$ and $\kappa$).
\end{theo}

One can write $\aniso(\bar\varphi)$ explicitly, see below. This theorem
is proved using a rather direct generalization of Milnor's argument
who considers the constant curvature, $2$-dimensional case. 
\begin{rema*}\begin{enumerate}
\item It is interesting to see that
the sign of $\kappa$ has such an influence on the optimal map: when $\kappa>0$ the
best map is isometric along rays issued from the center, and increases distances
in the orthogonal directions, while when $\kappa<0$ the best map induces
an isometry on the boundaries but contracts the radial rays. Of course, when $\kappa=0$
all $\sigma$-contracting azimuthal maps are equivalent up to an homothety,
and as long as $\sigma_1\le1\le\sigma_2$ their anisometries are equal.
\item The hypothesis on $N$ can be relaxed thanks to the generalized
Günther inequality proved with Greg Kuperberg \cite{KK:Gunther}.
In particular, $K_N\le \kappa$ can be replaced by mixed curvatures bounds like
\[K_N\le\rho \quad\mbox{and}\quad \Ric_N \le (n-1)\kappa-n\rho\]
see Section \ref{sec:candle}
and Theorem \ref{theo:general} for the most general hypothesis
and the above reference for various classical assumptions that imply this
general hypothesis.
\item The precise expression of $A_1$ is given page \pageref{eq:A1negative}.
In many cases one can adapt 
the result and its proof to larger $\alpha$ but
we favored clarity over exhaustivity. For example,
what happens for $\alpha$ close to $\frac{\pi}{\sqrt{\rho}}$
is that the boundary of $B_\rho(\alpha)$ becomes very small, and  
one can improve the equidistant azimuthal map by making it dilating along the rays.
\end{enumerate}
\end{rema*}

We shall then consider maps satisfying special conditions. Two prominent examples
are \emph{volume-preserving} maps and \emph{conformal} maps. In cartography, 
both make sense: area is obviously a relevant geographic information, and for many 
historical uses
(e.g. navigation) measurement of angles on the map have been needed. Moreover,
asking a map to be conformal means
that zooming into the map will decrease arbitrarily the distortion of
a smaller and smaller region. We therefore ask whether in general,
asking $\varphi$ to be volume-preserving or conformal increases the anisotropy lower bound
by much. 

In the theorems below, we shall make the assumption that $N$ satisfies the best 
isoperimetric inequality holding on $X_\kappa$, meaning that
for all smooth $\Omega\subset N$, 
\[\Vol(\partial\Omega)\ge I_\kappa(\Omega)\]
where $I_\kappa$ is the isoperimetric profile of $X_\kappa$ defined
by
\[I_\kappa(V) = \inf_{\Omega\subset X_\kappa}\{\Vol(\partial \Omega) \,|\, \Vol(\Omega)=V \}.\]
This assumption can be replaced by $K_N\le \kappa$ in some cases.

One says that $n$ is a \emph{Hadamard manifold} if $K_N\le 0$
and $N$ is simply connected; it is conjectured that all Hadamard manifolds
satisfy the isoperimetric inequality of $X_\kappa$ whenever $K_N\le \kappa$, but
this conjecture has only been proved in a handful of cases:
when $n=2$ \cite{Weil,Aubin}, $n=3$ \cite{Kleiner},
$(n=4, \kappa= 0)$ \cite{Croke} and $(n=4,\kappa<0)$ for small enough domains
\cite{KK:petitprince}. Moreover, the similar conjecture when $\kappa>0$
holds in dimension $n=4$ for uniquely geodesic domains \cite{KK:petitprince}.
When $n=4$ the curvature assumption
can generally be relaxed as for Theorem \ref{theo:intro-general}, see 
Section \ref{sec:candle} below and \cite{KK:petitprince}.

This means that in most dimensions,
our results below hold under a curvature assumption only
conditionally to a strong conjecture; but note that even in the case when 
$N= X_\kappa$ these results are new.

\begin{theo}[Volume-preserving maps]\label{theo:intro-area}
Assume $\Ric_M\ge\rho$, $N$ satisfies the best isoperimetric inequality
holding on $X_\kappa$ for some $\kappa <\rho$, and $\alpha\le \inj(x)$.

Then any volume-preserving map
$\varphi:B(x,\alpha)\subset M \to N$ satisfies
\[\aniso(\varphi) \ge \aniso(\bar\varphi)\]
where $\bar \varphi$ is the unique volume-preserving azimuthal map
$B_\rho(\alpha)\to X_\kappa$.

Assume further that the only domains in $N$ satisfying the equality case
in the isoperimetric inequality are balls isometric to geodesic balls in $X_\kappa$.
Then whenever $\aniso(\varphi)=\aniso(\bar\varphi)$, the domain of 
$\varphi$ has constant curvature $\rho$ and its range is isometric to a
constant curvature ball $B_\kappa(R(\alpha))$. However, there are uncountably
many different maps achieving equality.
\end{theo}

\begin{rema*}
Here we have put little restriction on $\alpha$ (we only restrict
it below the injectivity radius at $x$ for simplicity), but in fact stronger
restriction can appear when one wants to apply the result. Indeed, if one is only able
to show that small enough domains of $N$ satisfy the desired isoperimetric inequality,
then one can still use Theorem \ref{theo:intro-area} for small enough $\alpha$:
then a map $B(x,\alpha)\to N$ either has a small image, or a large $\sigma_2$.
\end{rema*}

\begin{theo}[Conformal maps]\label{theo:intro-conformal}
Assume $\Ric_M\ge\rho$, $N$ satisfies the best isoperimetric inequality
holding on $X_\kappa$ for some $\kappa <\rho$, and $\alpha\le A_3(M,N)$
where $A_3(M,N)$ is an explicit positive constant.

Then any conformal map
$\varphi:B(x,\alpha)\subset M \to N$ satisfies
\[\aniso(\varphi) \ge \aniso(\bar\varphi)\]
where $\bar \varphi$ is:
\begin{itemize}
\item the conformal azimuthal map $B_\rho(\alpha)\to X_\kappa$
      with $R'(0)=1$ when $\kappa\ge 0$,
\item the conformal azimuthal map $B_\rho(\alpha)\to X_\kappa$ 
      that induces an isometry on the boundaries
      when $\kappa<0$.
\end{itemize}

Assume further that the only domains in $N$ satisfying the equality case
in the isoperimetric inequality are balls isometric to geodesic balls in $X_\kappa$.
Then whenever $\aniso(\varphi)=\aniso(\bar\varphi)$,
the maps $\varphi$ and $\bar\varphi$ are conjugated
by isometries (in particular, the domain and range of $\varphi$ have constant
curvature $\rho$ and $\kappa$), except that when $\kappa=0$ one can compose
$\bar\varphi$ with any homothety such that we still have $\sigma_1\le 1\le\sigma_2$,
and still get an optimal map. 
\end{theo}

\begin{rema*}
We shall see that $A_3$ can in fact be chosen independently of $N$
(but depending on $\kappa$). Moreover, when $\kappa\le 0$
we can take $A_3=\inj(x)$.
\end{rema*}

Conformal maps are rare in higher dimension, so we also
tackle quasi-conformal maps, whose angular distortion is controlled. 
Recall that a smooth map $\varphi$ is said to be $Q$-quasiconformal if at
each point $x$ in its domain, we have $\dist(D\varphi_x)\le Q$, i.e. its infinitesimal
distortion is uniformly bounded; conformal maps are precisely the $1$-quasiconformal maps.
\begin{theo}\label{theo:intro-quasiconf}
Assume $\Ric_M\ge\rho$, $N$ satisfies the best isoperimetric inequality
holding on $X_\kappa$ for some $\kappa <\rho$, let $Q$ be a number greater than $1$
and assume $\alpha\le A_4(M,N,Q)$ where $A_4(M,N,Q)$ is some positive constant.

Then any $Q$-quasiconformal map
$\varphi:B(x,\alpha)\subset M \to N$ satisfies
\[\aniso(\varphi) \ge \aniso(\bar\varphi)\]
where $\bar \varphi$ is an explicit $Q$-conformal azimuthal map, which is
$C^1$ but not $C^2$.

Assume further that the only domains in $N$ satisfying the equality case
in the isoperimetric inequality are balls isometric to geodesic balls in $X_\kappa$.
Then whenever $\aniso(\varphi)=\aniso(\bar\varphi)$,
the maps $\varphi$ and $\bar\varphi$ are conjugated
by isometries (in particular, the domain and range of $\varphi$ have constant
curvature $\rho$ and $\kappa$), except that when $\kappa=0$ one can compose
$\bar\varphi$ with any homothety such that we still have $\sigma_1\le 1\le\sigma_2$,
and still get an optimal map. 
\end{theo}

\begin{rema*}
Here the constant $A_4$ is less explicit than in the other result,
but it is still perfectly constructive. Moreover we shall see
that when $\kappa\le 0$, we can take $A_4=\inj(x)$.
\end{rema*}

It is also interesting to compare what we obtain from the above inequalities
when $\alpha$ is small.
\begin{coro}\label{theo:intro-infinitesimal}
If $\Ric_M\ge\rho$ and $K_N\le\kappa$, any map
$\varphi:B(x,\alpha)\subset M \to N$ satisfies
\[\aniso(\varphi) \ge \frac16 (\rho-\kappa) \alpha^2 + o(\alpha^2).\]
If $\varphi$ is conformal, then
\[\aniso(\varphi) \ge \frac14 (\rho-\kappa) \alpha^2 + o(\alpha^2).\]
If $\varphi$ is volume-preserving, then
\[\aniso(\varphi) \ge \frac{n}{2(n+2)}(\rho-\kappa) \alpha^2 + o(\alpha^2).\]
\end{coro}

\begin{rema*}\begin{enumerate}
\item In this Corollary, one can easily replace the curvature assumptions
by scalar curvature bounds, since only small balls are considered. Note
that the isoperimetric inequality needed in Theorems \ref{theo:intro-area}
and \ref{theo:intro-conformal} has been proved to be true for small enough domains
under the curvature assumption $K_N<\kappa$ (or even $K_N\le\kappa$ in some cases) 
by Johnson and Morgan \cite{Johnson-Morgan} and under $\Scal_N<\kappa$  by Druet \cite{Druet}.
To obtain a Taylor series, these strict assumptions are sufficient
(but then the remainder term cannot be made explicit).
\item In all our results, one consider maps from the higher-curvature manifold to the lower-curvature
one. These results imply similar estimates for maps $\varphi : B(y,\alpha)\subset N\to M$, 
because either such a map contracts some distances by much (hence has large anisometry),
or its image contains a ball of radius bounded below, allowing us to apply the results
above to $\varphi^{-1}$. However, the estimates one gets that way are certainly not
sharp, and we do not know whether $\bar\varphi^{-1}$ is optimal
in any of the situation treated above; it seems that even the case of a map from a ball 
in the plane to a round $2$-sphere is open. One might want to perturb the equidistant 
azimuthal map to enlarge the boundary of its image, so as to limit the distortion
along the boundary. It is not clear whether this can be achieved without increasing distortion 
too much anywhere else.
\end{enumerate}
\end{rema*}

% Let us finish this introduction with a question that generalizes a conjecture
% of Milnor \cite{Milnor}. Let us define the (Euclidean) distortion of a domain
% in a manifold as the infimum of distortions of maps from the domain to Euclidean space.
% Milnor conjectures that among the convex domains on $\mathbb{S}^2$ of given volume, the
% geodesic sphere has largest distortion. If proven true, this would gives a uniform upper bound
% on the distortion of all convex spherical domains. Similarly, we can ask the following.
% \begin{ques}
% Is it true that among all convex domains $D$ of all manifolds $M$ such that
% $\Ric_M\ge 1$  
% \end{ques}

%%%%%%%%%%%%%%%%%%%%%%%%%%%%%%%%%%%%%%%%%%%%%%%%%%%%%%%%%%%%%%%
\subsubsection*{Organization of the paper}
Next section gives notations and some background. We prove our 
main results in the following three sections (general maps, then volume-preserving maps,
then conformal and quasi-conformal maps). 
The technique we use in the conformal and quasiconformal cases turns
out to have been used by Gromov to generalize the Schwarz-Pick-Ahlfors lemma.
In the final Section \ref{sec:appendix}, we shall state and prove a result of this flavor that seems
not to be in the literature (but certainly is in its topological closure).

%%%%%%%%%%%%%%%%%%%%%%%%%%%%%%%%%%%%%%%%%%%%%%%%%%%%%%%%%%%%%%%
\subsubsection*{Acknowledgments}

It is a pleasure to thank Charles Frances, Étienne Ghys and Pierre Pansu for
interesting discussions related to the content of the present article.

%%%%%%%%%%%%%%%%%%%%%%%%%%%%%%%%%%%%%%%%%%%%%%%%%%%%%%%%%%%%%%%
%%%%%%%%%%%%%%%%%%%%%%%%%%%%%%%%%%%%%%%%%%%%%%%%%%%%%%%%%%%%%%%
\section{Toolbox}

%%%%%%%%%%%%%%%%%%%%%%%%%%%%%%%%%%%%%%%%%%%%%%%%%%%%%%%%%%%%%%%
\subsection{Notations}
Let $X_\kappa$ be the model space of curvature $\kappa$ and
dimension $n$, i.e. a round sphere when $\kappa>0$,
the Euclidean space when $\kappa=0$, and a hyperbolic space
when $\kappa<0$.

We denote by $B_M(x,t)$ (respectively $S_M(x,t)$) the geodesic closed ball 
(respectively sphere) of radius $t$ and center $x$ in $M$.
When there is no ambiguity, we let
$B(t)=B_M(x,t)$ and $S(t)=S_M(x,t)$.
To simplify notation, we set $B_\kappa(t)$ (respectively $S_\kappa(t)$)
for any geodesic closed ball (respectively sphere) of radius
$t$ in $X_\kappa$. 

The volumes of manifolds, submanifolds and domains shall be denoted either
by $\Vol(\cdot)$ or $|\cdot|$.
We let $\omega_{n-1}=|S_0(1)|$ be the $(n-1)$-dimensional volume 
of the unit sphere in $X_0=\mathbb{R}^n$.

When there is no ambiguity, $\sigma_i$ shall denote $\sigma_i(\varphi)$.

When $x$ is a point in a manifold and $u$ a tangent vector at $x$,
we let $\gamma_u(t)=\exp_x(tu)$ be the time $t$ of the geodesic 
issued from $x$ with velocity $u$.

We shall denote by $T^1 M$ the unit tangent bundle of a Riemannian manifold $M$,
by $\inj(x)$ the injectivity radius at $x\in M$ and by $\inj(M)$ the injectivity
radius of $M$.

%%%%%%%%%%%%%%%%%%%%%%%%%%%%%%%%%%%%%%%%%%%%%%%%%%%%%%%%%%%%%%%
\subsection{Geometry of model spaces}

The model spaces $X_\kappa$ are well understood, let us recall a few
facts about them.

%%%
\subsubsection{Trigonometric functions}

It will be convenient to use the functions $\sin_\kappa$ defined by
\[ \sin_\kappa(a) = \begin{cases}
  \frac{\sin(\sqrt{\kappa} a)}{\sqrt{\kappa}} &\mbox{if } \kappa>0 \\
  a &\mbox{if } \kappa=0 \\
  \frac{\sinh(\sqrt{-\kappa} a)}{\sqrt{-\kappa}} &\mbox{if } \kappa<0
\end{cases}\]
We then set 
\[\cos_\kappa(a):= \sin_\kappa'(a) = \begin{cases}
  \cos(\sqrt{\kappa} a) &\mbox{if } \kappa>0 \\
  1 &\mbox{if } \kappa=0 \\
  \cosh(\sqrt{-\kappa} a) &\mbox{if } \kappa<0
\end{cases}\]
and
\[\tan_\kappa(a) := \frac{\sin_\kappa(a)}{\cos_\kappa(a)} = \begin{cases}
  \frac{\tan(\sqrt{\kappa} a)}{\sqrt{\kappa}} &\mbox{if } \kappa>0 \\
  a &\mbox{if } \kappa=0 \\
  \frac{\tanh(\sqrt{-\kappa} a)}{\sqrt{-\kappa}} &\mbox{if } \kappa<0
\end{cases}\]
We shall also use occasionally
\[\arctan_\kappa(x) := \tan_\kappa^{-1}(x) = \begin{cases}
  \frac{\arctan(\sqrt{\kappa} x)}{\sqrt{\kappa}} &\mbox{if } \kappa>0 \\
  x &\mbox{if } \kappa=0 \\
  \frac{\argth(\sqrt{-\kappa} x)}{\sqrt{-\kappa}} &\mbox{if } \kappa<0
\end{cases}\]
and we have the derivatives $\tan_\kappa'=1+\kappa \tan_\kappa^2$
and $\arctan_\kappa'(x)=\frac1{1+\kappa x^2}$.

A trigonometric formula that will prove useful is
\[\sin_\kappa(2\arctan_\kappa x) = \frac{2x}{1+\kappa x^2}.\]

We shall need the following Taylor series:
\begin{align*}
\sin_\kappa(t) &= t -\frac\kappa6 t^3 + \frac{\kappa^2}{24} t^5 + O(t^5) \\
\cos_\kappa(t) &= 1-\frac\kappa2 t^2 + \frac{\kappa^2}{24} t^4 + O(t^6) \\
\tan_\kappa(t) &= t + \frac\kappa3 t^3 +\frac{2\kappa^2}{15} t^5 + O(t^5) \\
\arctan_\kappa(x) &= x -\frac\kappa3 x^3 + \frac{\kappa^2}5 x^5 + 0(x^7).
\end{align*}

%%%
\subsubsection{Volumes}

Let $x$ be a point on $X_\kappa$, $t$ be a positive real and $u$ be a unit
tangent vector at $x$; then setting $y=\exp_x(tu)$ we can express the volume
measure $\dd y$ on $X_k$ by the formula
\[\dd y = \sin_\kappa^{n-1}(t) \,\dd t \,\dd u\]
where $\dd t$ is Lebesgue measure on $[0,+\infty)$ and $\dd u$ is the volume measure
on the unit tangent sphere $T_x^1 X_\kappa$ naturally identified with the unit round
sphere $S^{n-1}$.

In this volume formula, one can decompose the density into factors $1$ (in the direction of the ray
from the pole) and $\sin_\kappa(t)$ (in the $n-1$ orthogonal directions).
This shows that up to isometry there exists
exactly one azimuthal, conformal map $\varphi$ from the ball $B_\kappa(x,\alpha)$
such that
$D\varphi_x$ is a homothety of ratio $\sigma$ (i.e. $R'(0)=\sigma$), whose
distance function is driven by the following differential equation:
\[R'(t) = \frac{\sin_\kappa(R(t))}{\sin_\rho(t)}.\]

Moreover the $(n-1)$-dimensional volume of a geodesic sphere $S_\kappa(t)$ of radius
$t$ is
\[A_\kappa(t) := |S_\kappa(t)| = \omega_{n-1} \sin_\kappa^{n-1}(t)\]
where $\omega_{n-1}$ is the volume of $S^{n-1}$; when $\kappa>0$ we
only consider $t$ below the conjugate radius $\pi/\sqrt{\kappa}$.
We also name the volume of geodesic balls of $X_\kappa$: 
\[V_\kappa(t) := |B_\kappa(t)| = \omega_{n-1} \int_0^t \sin_\kappa^{n-1}(s) \,\dd s.\]

Given $\rho$ and $\kappa$, there is exactly one volume-preserving azimuthal map,
defined by the distance function
\[R(t) = V_\kappa^{-1}(V_\rho(t)).\]
That $R$ is as above is clearly necessary for an azimuthal map to be volume-preserving, but 
the local volume formula shows that it is also sufficient.

It is known that in $X_\kappa$ the least perimeter domains of given domains are balls,
so that the isoperimetric profile of $X_\kappa$ is given by
\[I_\kappa(V_\kappa(t))=A_\kappa(t).\]
Note that the lesser is $\kappa$, the greater is $I_\kappa$ and the
more stringent is the corresponding isoperimetric inequality.

Using the above Taylor series, we get:
\begin{align*}
A_\kappa(t) &= \omega_{n-1} t^{n-1} \left(1-\frac{(n-1)\kappa}6 t^2+ O(t^4)\right) \\
V_\kappa(t) &= \frac{\omega_{n-1}}n t^n \left(1 -\frac{n(n-1)\kappa}{6(n+2)} t^2 + O(t^4)\right) \\
I_\kappa(v) &= n^{\frac{n-1}n} \omega_{n-1}^{\frac1n} v^{\frac{n-1}n}
  -\frac{(n-1)\kappa}{2(n+2)}\cdot\frac{n^{\frac{n+1}n}}{\omega_{n-1}^{\frac1n}} 
  v^{\frac{n+1}n} +O(v^{\frac{n+3}n}). 
\end{align*}

%%%%%%%%%%%%%%%%%%%%%%%%%%%%%%%%%%%%%%%%%%%%%%%%%%%%%%%%%%%%%%%
\subsection{Candle functions and comparison}\label{sec:candle}

To study anisometry of maps under curvature bounds of the domain and range,
we will need some tools of comparison geometry, relating
the geometry of $M$ and $N$ to the geometry of $X_\rho$ and $X_\kappa$.
We will notably rely on Bishop and Günther's inequality, which in their 
common phrasing compare volume of balls. It will be useful to discuss
their more general form, which is about comparing Jacobians of exponential maps.

Given a point $x\in M$, a vector $u\in T^1_xM$ and a real number $t$,
let $y=\exp_x(tu)$ and  define the \emph{candle function} $j_x(tu)$ 
as a normalized Jacobian of the exponential map by
\[\dd y = j_x(tu) \,\dd t \,\dd u\]
where $\dd y$
denotes the Riemannian volume and $\dd u$ is the spherical
measure on $T^1_xM$. 

In the case of $X_\kappa$, this function 
does not depend on $x$ nor on $u$ and is equal to $\sin_\kappa^{n-1}(t)$.

\begin{defi}
The manifold $M$ is said to satisfy the \emph{candle} condition
$\Candle(\kappa,\ell)$ if for all $x$, $u$ and all $t\le\ell$
it holds
\[j_x(tu)\ge \sin_\kappa^{n-1}(t).\]

The manifold $M$ is said to satisfy the \emph{logarithmic candle derivative} condition
$\LCD(\kappa,\ell)$ if for all $x$, $u$ and all $t\le\ell$ it holds
\begin{equation}
\frac{j'(t)}{j(t)}\ge \frac{s'(t)}{s(t)}
\label{eq:LCD}
\end{equation}
where $j(t):=j_x(tu)$ and $s(t):=\sin_\kappa^{n-1}(t)$.
\end{defi}

The name ``candle condition'' is motivated by the fact
that $j_x$ describes the fade of the light of a candle (or of the
gravitational fields generated by a punctual mass) in $M$.

By integration, $\Candle(\kappa,\ell)$ implies that spheres and ball
of radius at most $\ell$ have volume at least as large as the volume
of the spheres and balls of equal radius in $X_\kappa$.

The candle condition is an integrated version of the
logarithmic candle derivative condition, which itself 
follows for $\ell=\inj(N)$ from the sectional curvature condition
$K\le \kappa$: this is known as G\"unther's theorem,
see \cite{GHL}. With Greg Kuperberg, we proved in \cite{KK:Gunther}
that it also follows from a weaker curvature bound, involving
the ``root-Ricci curvature''. In particular, we proved that
manifolds satisfying a relaxed bound on $K$ and a suitably strengthened bound on $\Ric$
still satisfy a LCD condition, and therefore a Candle one.

The strong form of Bishop's Theorem is that the reversed inequality in 
\eqref{eq:LCD} holds under the curvature lower bound $\Ric\ge \kappa$
(for $\ell$ the conjugate time of $X_\kappa$). The corresponding comparison
on the volumes of spheres and balls follow and are also referred to 
as Bishop's inequality.

We shall establish Theorem \ref{theo:intro-general} using the comparison of
spheres; the assumption $K_N\le \kappa$ can therefore be relaxed
to $\Candle(\kappa,\ell)$ where $\ell$ can be taken to be e.g. $\infty$
when $N$ is a Hadamard manifold or chosen suitably otherwise, see the proof below.

%%%%%%%%%%%%%%%%%%%%%%%%%%%%%%%%%%%%%%%%%%%%%%%%%%%%%%%%%%%%%%%
\subsection{Volume of ellipsoids and hyperplanes}

A couple of our arguments will rely on a simple and classical lemma,
which we state and prove for the sake of completeness.

Let $q$ be a scalar product in Euclidean space of dimension $n$, endowed
with the standard inner product $\langle\cdot,\cdot\rangle$. We shall
denote by $\sigma_1(q)$ and $\sigma_2(q)$ the largest, respectively smallest
numbers such that
\[\sigma_1(q)\langle u, u\rangle \le q(u,u) \le \sigma_2(q)\langle u,u\rangle 
  \quad\forall u\in\mathbb{R}^n\]
and say that $q$ is at most $Q$-distorted if $\sigma_2/\sigma_1\le Q$.
We shall also denote by $|q|$ the determinant of $q$, that is the ratio of the volume of 
its unit ball to the volume of Euclidean unit ball (both volumes computed with respect
to the Lebesgue measure associated to $\langle\cdot,\cdot\rangle$).

\begin{lemm}\label{lemm:ellipsoids}
Let $q_0$ be the restriction of $q$ to any hyperplane. Then we have
\[|q|\ge |q_0| \sigma_1(q) \]
and
\[|q|\ge \frac1Q |q_0|^{\frac n{n-1}}.\]
There is equality in this second inequality if and only if $q$ has eigenvalues
$\lambda$  and $Q\lambda$, with respective multiplicity $1$ and $n-1$,
and the hyperplane defining $q_0$ is the $Q\lambda$ eigenspace of $q$.
\end{lemm}

\begin{proof}
Let $\lambda_1\ge\dots\ge\lambda_n$ be the eigenvalues of $q$ and
$\mu_1\ge\dots\mu_{n-1}$ be the eigenvalues of $q_0$.
In particular, $\lambda_n=\sigma_1(q)$. Then Rayleigh
quotients show that $\mu_i\le \lambda_i$ for all $i<n$. It follows
\begin{align*}
|q| &= \lambda_1\lambda_2\dots\lambda_n \\
  &\ge |q_0| \lambda_n = |q_0| \sigma_1(q)
\end{align*}
But by the distortion bound, we have
\[Q\sigma_1\ge\lambda_1\ge\dots\ge\lambda_n\ge\sigma_1\]
so that $(Q\sigma_1)^{n-1}\ge |q_0|$ and
\[\sigma_1 \ge \frac{|q_0|^{\frac1{n-1}}}Q\]
and the desired inequality follows. The equality case is straightforward.
\end{proof}

%%%%%%%%%%%%%%%%%%%%%%%%%%%%%%%%%%%%%%%%%%%%%%%%%%%%%%%%%%%%%%%
%%%%%%%%%%%%%%%%%%%%%%%%%%%%%%%%%%%%%%%%%%%%%%%%%%%%%%%%%%%%%%%
\section{General maps}

Assume that $\Ric_M \ge \rho$ and that $N$ satisfies $\Candle(\kappa,\ell_0)$
for some $\ell_0$ (on which we shall put some restriction latter on).

Let $\varphi: B_M(x,\alpha) \to N$ be a diffeomorphism on its image,
where $\alpha<\inj(x)$, the injectivity radius of $M$ at $x$.

In what follows, we shall assume bounds involving $1/\sqrt{\kappa}$: our convention is
that this number is $+\infty$ whenever $\kappa\le 0$.
\begin{lemm}\label{lemm:general}
If $\sigma_2(\varphi) \alpha \le \ell_0 \le \frac{\pi}{2\sqrt{\kappa}}$, we have
\[ \sigma_2(\varphi) \ge \frac{\sin_\kappa(\sigma_1(\varphi)\alpha)}{\sin_\rho(\alpha)}.\]
\end{lemm}

The proof is a mere generalization of Milnor's argument in \cite{Milnor}.

\begin{proof}
Denote by $S(\alpha)$ the geodesic sphere of center $x$ and radius $\alpha$.
Bishop's inequality ensures that $|S(\alpha)|\ge \omega_{n-1} \sin_\rho^{n-1}(\alpha)$.
On the other hand, $\varphi(S(\alpha))$ encloses the ball of $N$ of radius
$\sigma_1 \alpha$ centered at $y=\varphi(x)$. Given a unit vector $u\in T_y N$,
let $\ell(u)$ be the first time at which $\gamma_u$ hits $\varphi(S(\alpha))$,
and $\beta(u)$ be the angle between $\dot\gamma_u(\ell(u))$ and the outward
normal to $\varphi(S(\alpha))$. We have $\ell(u)\ge \sigma_1 \alpha$ and
obviously $\cos(\beta(u))\le 1$. Moreover when $\kappa>0$, we have
\[\ell(u) \le \sigma_2 \alpha \le \frac{\pi}{2\sqrt{\kappa}}\]
so that the comparison candle function $\sin_\kappa^{n-1}$ is increasing
on $[0,\ell(u)]$.

Then, letting $j$ be the candle function of $N$
at $x$ and $\dd u$ be the usual measure on the unit Riemannian sphere, we get 
\begin{align*}
|\varphi(S(\alpha))| &\ge \int_{T_y^1 N} \frac{j_u(\ell(u))}{\cos \beta(u)} \,\dd u \\
  &\ge \int_{T_y^1 N} \sin_\kappa^{n-1}(\ell(u)) \,\dd u \\
  &\ge \omega_{n-1} \sin_\kappa^{n-1}(\sigma_1 \alpha).
\end{align*}

There is at least one point on $S(\alpha)$ at which the Jacobian of the restriction
of $\varphi$ to $S(\alpha)$ is at least 
\[\frac{|\varphi(S(\alpha))|}{|S(\alpha)|} \ge 
  \frac{\sin_\kappa^{n-1}(\sigma_1(\varphi)\alpha)}{\sin_\rho^{n-1}(\alpha)}\]
and the lemma follows.
\end{proof}

Let us now define the $\alpha$-bound $A_1$.
\begin{defi*}
Let $A_1=A_1(M,N)$ be the greatest number such that for all $\alpha\le A_1$
we have $\alpha\le \inj(M)$, if $\kappa> 0$:
\begin{equation}
\alpha\frac{\sin_\rho(\alpha)}{\sin_\kappa(\alpha)} \le 
  \min\left(\inj(N),\frac{\pi}{2\sqrt{\kappa}}\right)
\end{equation}
and if $\kappa\le 0$
\begin{equation}\frac{\alpha^2}{\sin_\kappa^{-1}\circ \sin_\rho(\alpha)} \le \inj(N).
\label{eq:A1negative}
\end{equation}
\end{defi*}

\begin{rema*}\begin{enumerate}
\item $A_1$ depends on $M$ and $N$ only through their curvature/candle bounds $\rho$ and $\kappa$
  and their injectivity radii,
\item if we do not insist on a uniform bound over possible centers, we can replace $\inj(M)$
  by $\inj(x)$,
\item if $N$ is a Hadamard manifold, then $A_1=\inj(M)$ (or $\inj(x)$),
\item if $\inj(M)$ and $\inj(N)$ are large enough, when $\kappa>0$ we have
  $A_1\ge \frac{\pi}{2\sqrt{\rho}}$ and $A_1\to \frac{\pi}{\sqrt{\rho}}$
  when $\kappa\to 0$,
\item in some cases (e.g. when one can apply Klingenberg's Theorems,
  see \cite{Cheeger-Ebin} Theorems 5.9 and 5.10),
  the curvature bound on $N$ is sufficient to get an estimate on
  $\inj(N)$, and therefore to get a bound $A'_1$ that does not depend on the injectivity
  radius of the range.
\end{enumerate}\end{rema*}

\begin{theo}\label{theo:general}
Assume $\Ric_M \ge \rho$, $N$ satisfies $\Candle(\kappa,\inj(N))$ for some
$\kappa<\rho$ (e.g. $K_N\le \kappa$) and $\alpha\le A_1(M,N)$ defined above.
Let $\varphi:B_M(x\alpha) \to N$ be any smooth map.

If $\kappa\ge 0$ then 
\[\aniso(\varphi) \ge \log\frac{\sin_\kappa(\alpha)}{\sin_\rho(\alpha)}\]
and there is equality if and only if $\varphi$ is conjugated via isometries
to the equidistant azimuthal map from $B_\rho(\alpha)$ to $B_\kappa(\alpha)$
(in particular, $B_M(x,\alpha)$ and its image must have constant curvatures
$\rho$ and $\kappa$).

If $\kappa<0$ then, letting $\sigma_0=\sigma_0(\kappa,\rho,\alpha)$ be
the number in $(0,1)$ such that
$\sin_\kappa(\sigma_0\alpha) = \sin_\rho(\alpha)$, we have
\[\aniso(\varphi) \ge \log\frac{1}{\sigma_0}\]
and there is equality if and only if $\varphi$ is
the $\sigma_0$-contracting azimuthal map
from $B_\rho(\alpha)$ to $B_\kappa(\sigma_0\alpha)$
(in particular, $B_M(x,\alpha)$ and its image must have constant curvatures
$\rho$ and $\kappa$).
\end{theo}

Notice that $\sigma_0$ is the dilation coefficient that makes the volumes
of the spheres $S_\kappa(\sigma_0 \alpha)$ and $S_\rho(\alpha)$ coincide; 
it makes the $\sigma_0$-contracting azimuthal map a non-dilating map, i.e.
$\sigma_2(\bar\varphi)=1$ when $\kappa<0$.

\begin{proof}
We can assume $\sigma_2 \alpha$ is small enough to apply Lemma \ref{lemm:general},
otherwise the way we designed $A_1$ ensures that $\sigma_2$ is so large that
$\aniso(\varphi)$ is a least the claimed lower bound.

Let us start with the $\kappa\ge 0$ case.
From Lemma \ref{lemm:general} we have
\begin{equation}
\aniso(\varphi) \ge |\log \sigma_1 | 
  + \log \sin_\kappa(\sigma_1 \alpha) -\log \sin_\rho(\alpha).
\label{eq:generalA}
\end{equation}
The derivative of the right-hand side with respect to $\sigma_1$
is
\[-\frac1{\sigma_1} +\frac\alpha{\tan_\kappa(\sigma_1 \alpha)}<0\]
when $\sigma_1<1$
and
\[\frac1{\sigma_1} +\frac\alpha{\tan_\kappa(\sigma_1 \alpha)}>0\]
when $\sigma_1>1$. This shows that the right-hand side of \eqref{eq:generalA}
achieves its minimum
when $\sigma_1=1$, so that
\[\aniso(\varphi) \ge \log \sin_\kappa(\alpha) -\log \sin_\rho(\alpha).\]

In case of equality, one must have $\sigma_1=1$ and 
$\sigma_2=\sin_\kappa(\alpha)/\sin_\rho(\alpha)$, therefore there is
equality in Lemma \ref{lemm:general}. This forces $B_M(x,\alpha)$ and 
its image to have constant curvatures $\rho$ and $\kappa$ and $S(\alpha)$
must be mapped to the geodesic sphere of radius $\alpha$ and center $\varphi(x)$
in $N$. Since $\sigma_1=1$, $\varphi$ must then map  $S(a)$ to $S_\kappa(a)$
for all $a$. Each ray must be mapped to a curve
of length at its $\alpha$ that connects $\varphi(x)$ to the boundary of the image
of $\varphi$, therefore unit rays are mapped to unit rays. The whole map $\varphi$ then
depends only on its derivative, which must preserve the norms. It follows that
$\varphi$ is azimuthal equidistant, up to isometries.

In the $\kappa<0$ case, \eqref{eq:generalA} also holds but is not optimal anymore.
Indeed, the derivative of its right-hand side is positive both
when $\sigma_1>1$ and when $\sigma_1<1$ since $\tan_\kappa(x)\le x$.
But when $\sigma_1<\sigma_0$, the lower bound on $\sigma_2$ given by Lemma \ref{lemm:general}
is less than $1$. It follows
\[\aniso(\varphi) \ge \log\frac1{\sigma_0}\]
Which is achieved by $\bar\varphi$.
The case of equality is treated as above.
\end{proof}

\begin{coro}
In the above setting,
\[\aniso(\varphi) \ge \frac{\rho-\kappa}6 \alpha^2 + O(\alpha^4)\]
where the implied constant in the remainder term only depends on
the curvature bounds.
\end{coro}

%%%%%%%%%%%%%%%%%%%%%%%%%%%%%%%%%%%%%%%%%%%%%%%%%%%%%%%%%%%%%%%
%%%%%%%%%%%%%%%%%%%%%%%%%%%%%%%%%%%%%%%%%%%%%%%%%%%%%%%%%%%%%%%
\section{Area-preserving maps}

Let us now prove Theorem \ref{theo:intro-area} in the following form.

\begin{theo}\label{theo-area}
Assume $\Ric_M\ge\rho$, $N$ satisfies the best isoperimetric inequality
holding on $X_\kappa$, and $\alpha\le \inj(x)$. Then any volume-preserving map
$\varphi:B(x,\alpha)\subset M \to N$ satisfies
\[\aniso(\varphi) \ge \frac n{n-1} \log\frac{I_\kappa\circ V_\rho(\alpha)}{A_\rho(\alpha)}\]
and equality is achieved by the unique volume-preserving azimuthal map
$\bar\varphi : B_\rho(\alpha)\to X_\kappa$.

Assume further that the only domains in $N$ satisfying the equality case
in the isoperimetric inequality are balls isometric to geodesic balls in $X_\kappa$.
Then whenever $\aniso(\varphi)=\aniso(\bar\varphi)$, the domain of 
$\varphi$ has constant curvature $\rho$ and its range is isometric to a
constant curvature ball $B_\kappa(R(\alpha))$. However, there are uncountably
many different maps achieving equality.
\end{theo}

\begin{proof}
The key point is the following Lemma, which is a direct adaptation of
Theorem 3.5 in \cite{Johnson-Morgan}.
\begin{lemm}
Under the assumption $\Ric_M \ge \rho$ and for all $\kappa\le \rho$, we have
\[\frac{I_\kappa(|B(\alpha)|)}{|S(\alpha)|} \ge \frac{I_\kappa\circ V_\rho(\alpha)}{A_\rho(\alpha)}\]

If there is equality, then $B(\alpha)$ is isometric to $B_\rho(\alpha)$.
\end{lemm}

\begin{proof}[Proof of Lemma]
Setting $\delta_0 := \frac{|S(\alpha)|}{A_\rho(\alpha)}$, the strong form of Bishop's inequality
yields for all $t\le \alpha$
\[\delta_0 \le \frac{|S(t)|}{A_\rho(t)} \le 1.\]
By integration, it comes
\[\delta_0\le\delta_1 := \frac{|B(\alpha)|}{V_\rho(\alpha)} \le 1.\]
Since $I_\kappa$ is concave, we have $I_\kappa(\delta_1 V_\rho(\alpha)) 
\ge \delta_1 I_\kappa(V_\rho(\alpha))$; therefore
\[I_\kappa(|B(\alpha)|) \ge \delta_0 I_\kappa(V_\rho(\alpha)) 
   = |S(\alpha)|\frac{I_\kappa(V_\rho(\alpha))}{A_\rho(\alpha)}.\]

In case of equality, we must have $\delta_0=\delta_1$, which implies
$\delta_0=\delta_1=1$. The equality case in Bishop's inequality then implies
that $B(\alpha)$ is isometric to $B_\rho(\alpha)$.
\end{proof}

Now, since $\varphi$ is volume-preserving and $N$ satisfies an isoperimetric
inequality, we have
\[|\varphi(S(\alpha))|\ge I_\kappa(|B(\alpha)|) 
\ge |S(\alpha)|\frac{I_\kappa(V_\rho(\alpha))}{A_\rho(\alpha)}.\]
Then, there must be a point $x$ on $S(\alpha)$ such that the Jacobian of
$\varphi_{|S(\alpha)}$ is at least
$\frac{I_\kappa(V_\rho(\alpha))}{A_\rho(\alpha)}$
so that 
\[\sigma_2 \ge \left(\frac{I_\kappa(V_\rho(\alpha))}{A_\rho(\alpha)} \right)^{\frac1{n-1}}\]
but also, since $\jac \varphi_x = 1$, Lemma \ref{lemm:ellipsoids}
shows that a direction transverse to the boundary
must be contracted by $\varphi$ and
\[\sigma_1 \le \left(\frac{I_\kappa(V_\rho(\alpha))}{A_\rho(\alpha)} \right)^{-1}.\]
These two bounds combined imply the desired inequality on $\aniso(\varphi)$.

It is straightforward to see that the unique volume-preserving azimuthal map
$\bar\varphi : B_\rho(\alpha)\to X_\kappa$ realizes equality. Moreover,
if there is equality then there must be equality in the Lemma,
so that $B(\alpha)$ is isometric to $B_\rho(\alpha)$, and
there must be equality in the isoperimetric inequality on $N$.

However, $\bar\varphi$ is far from being the only optimal map: both $\sigma_1$
and $\sigma_2$ are realized on the boundary, and for all $t<\alpha$,
\[\sigma_1(\bar\varphi_{|B_\rho(t)}) > \sigma_1(\bar\varphi) \quad\mbox{and}\quad
     \sigma_2(\bar\varphi_{|B_\rho(t)}) < \sigma_2(\bar\varphi).\]
If we compose $\bar\varphi$ with any diffeomorphism of $B_\rho(\alpha)$
close to identity and supported on some $B_\rho(t)$, we get another optimal map.
\end{proof}

%%%%%%%%%%%%%%%%%%%%%%%%%%%%%%%%%%%%%%%%%%%%%%%%%%%%%%%%%%%%%%%
%%%%%%%%%%%%%%%%%%%%%%%%%%%%%%%%%%%%%%%%%%%%%%%%%%%%%%%%%%%%%%%
\section{Conformal and quasiconformal maps}

The following result is the heart of our results for quasiconformal
maps; it's formulation has been chosen to avoid repetition of
arguments while keeping as much flexibility as we shall need, and it is 
therefore rather technical.
\begin{theo}[Main quasiconformal inequality]
Assume $\varphi$ is a $Q$-quasiconformal maps from
$B_M(x,\alpha)$ to $N$, where $\Ric_M\ge \rho$ and
$N$ satisfies the isoperimetric inequality of $X_\kappa$.

Let $G_\kappa$ be the function defined by
\[G_\kappa(x)= \sin_\kappa(2\arctan_\kappa(x))=\frac{2x}{1+\kappa x^2}\]

If $\kappa>0$, assume further that the volume of the image of
$\varphi$ is not greater than the volume $\frac12|X_\kappa|$
of an hemisphere of curvature $\kappa$.

Then, for all $\beta<\alpha$ we have
\[\sigma_2(\varphi) \ge \frac{G_\kappa\left(\tan_\kappa(\frac{r(\beta)}2)\cdot
    \left(\frac{\tan_\rho(\frac\alpha2)}{\tan_\rho(\frac\beta2)}\right)^{\frac1Q}\right)}%
    {\sin_\rho(\alpha)}\]
where $r(\beta)$ is the radius of a ball in $X_\kappa$ that
has the same volume as the image of $B_M(x,\beta)$.
\end{theo}

The proof of this inequality follows a simple idea: at each time $t$,
the isoperimetric inequality forces the image of the sphere of radius $t$
to have large volume, and the quasiconformality then translate this
into a large increase in the volume of the image of the ball. These two effects
therefore amplify one another. At $t=\alpha$, we get a lower bound on $V(\alpha)$,
and using the isoperimetric inequality again we bound from below the perimeter
of the image of the $\alpha$-ball. Comparing with the perimeter of the ball, we get a lower
bound on $\sigma_2$.

\begin{proof}[Proof of the main quasiconformal inequality]
For convenience, for all $t\in(0,\alpha)$ set 
$V(t) = |\varphi(B(t))|$. In particular, $r(t)=V_\kappa^{-1}(V(t))$.

Using Hölder's inequality we get
\begin{align*}
V'(t) &= \int_{S(t)} |\jac\varphi|(y) \,\dd y \\
  &\ge \frac{\left(\int_{S(t)} |\jac\varphi|^{\frac{n-1}n}(y) \,\dd y\right)^{\frac{n}{n-1}}}{%
       |S(t)|^{\frac1{n-1}}}
\end{align*}
where $\jac\varphi(y)$ is the Jacobian of $\varphi$ at $y$. Let $\varphi_0$ be the restriction
of $\varphi$ along $S(t)$: using Lemma \ref{lemm:ellipsoids},
Bishop's inequality and the isoperimetric inequality on $N$ it comes
\begin{align}
V'(t) &\ge \frac{\frac{1}{Q} \left(\int_{S(t)} |\jac \varphi_0|(y) 
               \,\dd y\right)^{\frac{n}{n-1}}}{|S(t)|^{\frac1{n-1}}} \nonumber\\
  &\ge \frac{|\varphi(S(t))|^{\frac{n}{n-1}}}{Q|S_\rho(t)|^{\frac{1}{n-1}}} \nonumber\\
V'(t)  &\ge \frac{I_\kappa(V(t))^{\frac{n}{n-1}}}{%
    Q\omega_{n-1}^{\frac{1}{n-1}}\sin_\rho(t)}. \label{eq:proof-conf1}
\end{align}

Let $F=F_{\kappa,Q}$ be defined by
\[F \circ V_\kappa(t) = Q \log \tan_\kappa(t/2)\]
and let us compute $F'$:
\begin{align}
\frac{\dd }{\dd x}(F\circ V_\kappa(x)) &= \frac{\dd }{\dd x}(Q\log \tan_\kappa(x/2)) \nonumber\\
V_\kappa'(x) F'(V_\kappa(x)) &= \frac{Q}{\sin_\kappa(x)} \nonumber\\
F'(V_\kappa(x)) &= \frac{Q}{A_\kappa(x)\sin_\kappa(x)} \nonumber\\
  &= \frac{Q\omega_{n-1}^{\frac1{n-1}}}{A_\kappa(x)^{1+\frac1{n-1}}} \nonumber\\
  &= \frac{Q\omega_{n-1}^{\frac1{n-1}}}{(I_\kappa\circ V_\kappa(x))^{\frac{n}{n-1}}} \nonumber\\
F' &= \frac{Q\omega_{n-1}^{\frac1{n-1}}}{I_\kappa^{\frac{n}{n-1}}} \label{eq:proof-conf2}
\end{align}

From \eqref{eq:proof-conf1} and \eqref{eq:proof-conf2} it comes
\[F'(V(t))V'(t) \ge \frac1{\sin_\rho(t)}\]

As above, $\log(\tan_\rho(t/2))$ defines an antiderivative
of $1/\sin_\rho(t)$ and integrating we conclude
\begin{equation}
F(V(\alpha))-F(V(\beta)) \ge \log \frac{\tan_\rho(\alpha/2)}{\tan_\rho(\beta/2)}
\label{eq:qcinequality}
\end{equation}
since $F'$ is a positive function, $F$ is increasing and invertible, so that the above
inequality gives a lower bound on $V(\alpha)$;
using $|\varphi(S(\alpha))|\ge I_\kappa(V(\alpha))$ and proceeding
as in the proof of Theorem \ref{theo:general}, we get 
\begin{align*}
\sigma_2(\varphi) &\ge \left(\frac{I_\kappa(V(\alpha))}{|S(\alpha)|}\right)^{\frac1{n-1}} \\
  &\ge \frac{(I_\kappa\circ F^{-1})^{\frac1{n-1}}\left(F(V(\beta) + 
    \log(\frac{\tan_\rho(\alpha/2)}{\tan_\rho(\beta/2)})\right)}%
    {\omega_{n-1}^{\frac1{n-1}}\sin_\rho(\alpha)}
\end{align*}
(beware that exponent $\frac1{n-1}$ is a multiplicative power while exponent $-1$ stands for 
inverse function).

Now
\begin{align*}
(I_\kappa\circ F^{-1})^{\frac1{n-1}} 
  &= (I_\kappa\circ V_\kappa\circ (F\circ V_\kappa)^{-1})^{\frac1{n-1}} \\
  &= (A_\kappa \circ (F\circ V_\kappa)^{-1})^{\frac1{n-1}}\\
  &= \omega_{n-1}^{\frac1{n-1}} \sin_\kappa \circ (F\circ V_\kappa)^{-1}
\end{align*}
and the desired inequality follows from the identity
\[\sin_\kappa(2\arctan_\kappa x)=\frac{2x}{1+\kappa x^2}.\]
\end{proof}

\begin{rema*}
In the above proof, two small difficulties are hidden.
\begin{enumerate}
\item When $\kappa>0$, $I_\kappa$ is decreasing
beyond the volume of an hemisphere; this is why we assumed an upper bound on $V(\alpha)$.
\item When $\kappa<0$, $F$ has bounded image so that $F^{-1}$ is not defined on the whole
positive axis. Our proof shows that any quasiconformal map
must map  small balls to domains of relatively small volume (bounded in terms of 
$\alpha,\kappa,\rho$ and the radius of the considered ball), 
for otherwise the differential inequality on $V(t)$ would blow up in time less
than $\alpha$ and the map would not have compact image. This is the base to
a generalization of the Schwarz-Pick-Ahlfors lemma by Gromov, see the appendix.
\end{enumerate}
\end{rema*}

\begin{defi*}
Let $A_3=A_3(M,\kappa,\rho,n)$ be defined as the greatest real number such that
for all $\alpha \le A_3$ it holds
\begin{itemize}
\item $\alpha\le \inj(M)$,
\item if $\kappa>0$, 
\[\left(\frac{|X_\kappa|}{2 V_\rho(\alpha)}\right)^{\frac1n} \ge
  1+(\rho-\kappa)\frac{\tan_\rho^2(\alpha/2)}{1+\kappa\tan_\rho^2(\alpha/2)}.\]
\end{itemize}
\end{defi*}

Let us now prove Theorem \ref{theo:intro-conformal} which we restate as follows.
\begin{theo}\label{theo:conformal}
Assume $\Ric_M\ge\rho$, $N$ satisfies the best isoperimetric inequality
holding on $X_\kappa$ and $\alpha\le A_3(M,\kappa,\rho,n)$.
For the equality case below, assume further that
any domain $\Omega\subset N$ such that $|\partial\Omega|=I_\kappa(|\Omega|)$
is isometric to a geodesic ball in $X_\kappa$.
Let $\varphi:B(x,\alpha)\subset M \to N$ be a conformal map.

If $\kappa> 0$, then 
\[\aniso(\varphi) \ge \log\left(1+(\rho-\kappa)\frac{\tan_\rho^2(\alpha/2)}%
  {1+\kappa\tan_\rho^2(\alpha/2)}\right)\]
with equality when $\varphi$ is conjugated
by isometries the conformal azimuthal map $B_\rho(\alpha)\to X_\kappa$
with $R'(0)=1$.

If $\kappa=0$, then 
\[\aniso(\varphi) \ge \log\left(1+\rho\tan_\rho^2(\alpha/2)\right)\]
with equality when $\varphi$ is conjugated by isometries 
to a conformal azimuthal map $B_\rho(\alpha)\to X_0=\mathbb{R}^n$
with $R'(0)\le 1$ and $\sigma_2\ge 1$ (e.g. $R'(0)=1$).

If $\kappa<0$, then 
\[\aniso(\varphi) \ge \log \frac{-2\kappa \sin_\rho^2(\frac\alpha2)}%
  {\sqrt{1-\kappa\sin_\rho\alpha}-1}\]
with equality when $\varphi$ is conjugated
by isometries the conformal azimuthal map $B_\rho(\alpha)\to X_\kappa$ 
that induces an isometry on the boundaries (or, equivalently, that
preserves volumes along the boundary).
\end{theo}

\begin{rema*}
In the case $\kappa=0$ it is easy to compare the bound for general maps
and conformal ones. When $\alpha\to0$, this will be done more generally below;
when $\alpha\to\pi$, both lower bounds go to infinity, but in the conformal case
it does so twice as fast (after taking logs!) in the sense that
\[\frac{\aniso(\bar\varphi_c)}{\aniso(\bar \varphi)} \to 2\]
where $\bar\varphi$ and $\bar\varphi_c$ denote the optimal azimuthal maps
for radius $\alpha$ in the general and conformal cases, respectively.
Conformality thus appears to have a significant effect on anisometry.
\end{rema*}

\begin{proof}
The bound $A_3$ has been designed so that either $B(\alpha)$
is mapped to a domain so large that at some point $y$
$D\varphi_y$ itself must have anisometry at least equal to the claimed
bound, or the volume of $\varphi(B(\alpha))$ is at most $\frac12|X_\kappa|$
and we can use the main quasiconformal inequality with $Q=1$ and $\beta\to0$.

Since $r(\beta)\ge \sigma_1 \beta + o(\beta)$,
when $\beta\to0$ we have
\[\frac{\tan_\kappa(\frac{r(\beta)}2)}{\tan_\rho(\beta/2)} \ge \sigma_1 + o_\beta(1)\]
and we obtain
\[\sigma_2(\varphi) \ge \bar\sigma_2 
  :=\frac{G_\kappa(\sigma_1 \tan_\rho(\alpha/2))}{\sin_\rho(\alpha)}.\]
We would like to optimize in $\sigma_1$ the corresponding bound 
\[f(\sigma_1) := |\log \sigma_1 |+|\log \bar\sigma_2(\sigma_1)|\]
on $\aniso(\varphi)$.

For this, we observe that for all positive $\sigma$, the number $f(\sigma)$
is the anisometry of a conformal map with co-Lipschitz coefficient equal to $\sigma$.
For this, let 
\[\Phi_{\sigma,\kappa} : B_\rho(\alpha) \to X_\kappa\]
be the unique conformal azimuthal map such that $\sigma_1(\Phi_{\sigma,\kappa})=\sigma$
(i.e., its distance function satisfies $R_{\sigma,\kappa}'(0)=\sigma$).
Then following the proof of the main quasiconformal inequality with $Q=1$, we see
that all inequalities are equalities so that indeed $f(\sigma)=\aniso(\Phi_{\sigma,\kappa})$.

Moreover, if $\aniso(\varphi) = f(\sigma_1)$ (recall that $\sigma_1$ stands for $\sigma_1(\varphi)$) 
then we must have equality
in all inequalities in the proof of the main quasiconformal inequality, and this implies
that $\varphi$ and $\Phi_{\sigma_1,\kappa}$ are conjugated by isometries.

Observe that $\aniso(\Phi_{\sigma,\kappa})$ is decreasing with $\kappa$,
and increasing with $\sigma$ whenever 
\begin{equation}
\sigma\le1\le\sigma_2(\Phi_{\sigma,\kappa}).
\label{eq:range}
\end{equation}
It is clear that the minimum of $f(\sigma)$ occurs in this range. Observe further
that 
\[\aniso(\Phi_{\sigma,\kappa})=\aniso(\Phi_{\sigma/\lambda,\lambda^2 \kappa})\]
whenever $\sigma$ and $\sigma/\lambda$ both are in the range \eqref{eq:range},
since $\Phi_{\sigma/\lambda,\lambda^2 \kappa}$ is the composition
of $\Phi_{\sigma,\kappa}$ with a homothety of ratio $\lambda$.

When $\kappa>0$, if $\sigma<1$ is in the above range then we get
\[\aniso(\Phi_{\sigma,\kappa})=\aniso(\Phi_{1,\sigma^2 \kappa})>\aniso(\Phi_{1,\kappa})\]
so that $f(\sigma_1)\ge f(1)$ with equality if and only if $\sigma=1$.

When $\kappa<0$, if $1< \bar\sigma_2(\sigma) = \sigma_2(\Phi_{\sigma,\kappa})$, then
\[\aniso(\Phi_{\sigma,\kappa})=
  \aniso(\Phi_{\sigma/\bar\sigma_2,\bar\sigma_2^2 \kappa})>\aniso(\Phi_{\sigma/\bar\sigma_2,\kappa})\]
so that $f(\sigma)$ reaches its unique minimum for the value of $\sigma$
such that $\sigma_2(\Phi_{\sigma,\kappa})=1$.

When $\kappa=0$, $f$ is constant on the range \eqref{eq:range}.

% When $\kappa>0$, the bound assumed on $V(\alpha)$ ensures that
% $\sigma_1(\varphi)$ is in the range where $G_\kappa$ is increasing [[CHECK]];
% when $\kappa<0$, we must have $\sigma_1(\varphi)\tan_\rho(\alpha/2) \in (0,1/\sqrt{-\kappa})$,
% for otherwise $\varphi$ could not be defined and conformal up on the whole of 
% $B_M(x,\alpha)$ (see remark \ref{rema:quasiconformal}).

% Whenever $\sigma_1$ and $\bar\sigma_2$ are both lesser than $1$, $f'$ is negative;
% whenever they are both greater than $1$, $f'$ is positive; and of course
% $\bar\sigma_2(\sigma_1)\ge \sigma_1$ [[[CHECK]]]. We therefore only have to consider
% the case when $\sigma_1\le 1 \le \bar \sigma_2$.

Note that we could also have proceeded via calculus:
setting $x=\sigma_1\tan_\rho(\alpha/2)$
we then have
\begin{align*}
\frac{\dd }{\dd \sigma_1} e^{f(\sigma_1)} &=\frac1{\sigma_1^2} 
  \left(x G_\kappa'(x) - G_\kappa(x)\right) \\
  &= \frac{-4\kappa x^3}{\sigma_1^2(1+\kappa x^2)^2}
\end{align*}
Therefore, if $\kappa>0$ then $f$ has its only minimum when $\sigma_1=1$,
and if $\kappa<0$ then $f$ has its only maximum when $\bar\sigma_2(\sigma_1)=1$.
When $\kappa=0$, any value of $\sigma_1$ between this two cases yields the same
result.

We only have left to compute $\min f$.
When $\kappa\ge0$, we get
\[\aniso(\varphi) \ge \log\frac{G_\kappa (\tan_\rho(\frac\alpha2))}{\sin_\rho(\alpha)}.\]
Then, using 
\[\sin_\rho(\alpha) = \frac{2\tan_\rho(\frac\alpha2)}{1+\rho\tan_\rho^2(\frac\alpha2)}\]
we easily get the claimed inequality.

When $\kappa<0$, the minimum of $f$ is attained when $\bar\sigma_2=1$ and, therefore,
$\sigma_1$ is such that $G_\kappa(\sigma_1\tan_\rho(\frac\alpha2))=\sin_\rho(\alpha)$.
Since at this point we have $f(\sigma_1)=-\log(\sigma_1)$, we only have to inverse
$G_\kappa$ to get the desired inequality.
\end{proof}

\begin{coro}
If $\varphi$ is conformal, then
\[\aniso(\varphi) \ge \frac{\rho-\kappa}4 \alpha^2 + o(\alpha^2)\]
where the remainder depends on the curvature bounds and $N$.
\end{coro}

As mentioned above, we can even release the assumption on $M$ and $N$ to be $\Scal_M\ge\rho$ and
$\Scal_N\le\kappa$: the infinitesimal Bishop inequality holds true under
a scalar curvature bound, and so does the isoperimetric inequality as proved by Druet
\cite{Druet}.

\begin{theo}\label{theo:qc}
Assume $\Ric_M\ge\rho$ and $N$ satisfies the best isoperimetric inequality
holding on $X_\kappa$. For the equality case below, assume further that
any domain $\Omega\subset N$ such that $|\partial\Omega|=I_\kappa(|\Omega|)$
is isometric to a geodesic ball in $X_\kappa$.
Let $\varphi:B(x,\alpha)\subset M \to N$ be a $Q$-conformal map.

If $\kappa\ge 0$, let $\bar\varphi:B_\rho(\alpha)\to X_\kappa$ be the azimuthal map
whose distance function satisfies
\[R(t) = t \mbox{ when }t\le \beta \quad R'(t)=\frac{\sin_\kappa(R(t))}{Q\sin_\rho(t)} 
  \mbox{ when }t\ge \beta\]
where $\beta>0$ is such that
\[\frac{\sin_\kappa(\beta)}{Q\sin_\rho(\beta)}=1.\]
There is a positive number $A_4=A_4(M,N,Q)$ such that if $\alpha\le A_4$
then 
\[\aniso(\varphi) \ge \aniso(\bar\varphi)\]
and there is equality if and only if $\varphi$ and $\bar\varphi$ are conjugated by
isometries (except in the case $\kappa=0$ where the conjugating map on the range
can be a homothety).

If $\kappa <0$, let $\bar\varphi:B_\rho(\alpha)\to X_\kappa$ be the azimuthal map
whose distance function satisfies
\[R(t) = \sigma t \mbox{ when }t\le \beta \quad R'(t)=\frac{\sin_\kappa(R(t))}{Q\sin_\rho(t)} 
  \mbox{ when }t\ge \beta\]
where $\beta>0$ is such that
\[\frac{\sin_\kappa(\sigma \beta)}{Q\sin_\rho(\beta)}=\sigma\]
and $\sigma$ is such that $\sigma_2(\bar\varphi)=1$ (in particular,
$\bar\varphi$ induces an isometry on the boundary).
Then whenever $\alpha\le\inj(x)$ we have
\[\aniso(\varphi) \ge \aniso(\bar\varphi)\]
and there is equality if and only if $\varphi$ and $\bar\varphi$ are conjugated by
isometries.
\end{theo}

\begin{proof}
The proof follows exactly the same lines as the proof of Theorem \ref{theo:conformal},
using the quasi-conformal inequality with the chosen $\beta$ and $Q$.
Fixing $\rho$ and $\alpha$, for all given $\sigma$ and $\kappa$ we construct a comparison map
\[\bar\varphi_{\sigma,\kappa}:B_\rho(\alpha)\to X_\kappa\]
as in the conclusion of the Theorem. We remark that, restricted to $B_\rho(\beta)$,
$\bar\varphi$ is both the least anisometrical map and $Q$-conformal. Then for larger
radii, $A_4$ is designed so that either the anisometry bound is true or
the main quasiconformal inequality shows that $\bar\varphi_{\sigma_1(\varphi),\kappa}$
has lesser anisometry than $\varphi$.

To get the desired conclusion, we only have left to optimize the anisometry of
these comparison maps in $\sigma$. This does not differ from the conformal case.

We do not give
explicit values for the lower anisometry bounds, but they can be obtained explicitly
from the above computations (though probably not in closed form).
\end{proof}

\begin{rema}
It can be checked that in Theorem \ref{theo:qc}
the optimal azimuthal map is $C^1$ but not $C^2$ when $Q>1$.
\end{rema}

%%%%%%%%%%%%%%%%%%%%%%%%%%%%%%%%%%%%%%%%%%%%%%%%%%%%%%%%%%%%%%%
%%%%%%%%%%%%%%%%%%%%%%%%%%%%%%%%%%%%%%%%%%%%%%%%%%%%%%%%%%%%%%%
\section{Appendix: the generalized Schwarz-Pick-Ahlfors lemma}\label{sec:appendix}

The method used to prove the main quasiconformal inequality
was already used by Gromov \cite{Gromov-Lafontaine-Pansu,Gromov:pseudoh}
and Pansu \cite{Pansu} in relation with generalizations of Ahlfors lemma.

The classical Schwarz lemma says that a holomorphic
map $f:\Delta\to \Delta$ from the unit disc to itself, such that
$f(0)=0$ must satisfy $|f'(0)|\le 1$ and, in case of equality,
$f$ must be a rotation. Pick reinterpreted this result by endowing
the disc with its hyperbolic metric: the lemma then amounts to say that
any conformal map from $\mathbb{H}^2$ to $\mathbb{H}^2$ must be non-dilating
in the hyperbolic metric,
and if at any point its Jacobian has modulus $1$ then the map must be
a hyperbolic isometry. Then, Ahlfors extended this result to conformal
maps from a surface with curvature bounded below by $-1$ to a surface with
curvature bounded above by $-1$. This had a lasting impact on several fields
of mathematics. Among possible generalization to higher dimensions, one
that fits particularly well with the content of the present article is the following.

\begin{theo}\label{theo:Ahlfors}
Let $M$ and $N$ be complete manifolds of the same dimension (at least $2$)
with $\Ric_M\ge -1$ and $K_N\le -1$, and let
$\varphi:M\to N$ be a smooth conformal map.
If the Cartan-Hadamard conjecture holds, then $|\jac\varphi(x)|\le 1$ for all
$x\in M$, and if there is equality at any one point, then $\varphi$ lifts
to an isometry of the universal coverings $\tilde\varphi : \tilde M \to \tilde N$
(in particular, $M$ and $N$ have constant curvature $-1$).
\end{theo}

The current knowledge gives us the conclusion unconditionally when $N$
is the real hyperbolic space, and when the dimension is $2$ or $3$.

The  above result can hardly be considered new, but we could not find a written proof;
we therefore provide one.

\begin{proof}
Let $\tilde\varphi:\tilde M \to \tilde N$ be the lift to $\tilde N$
of the composition of the universal covering map $\pi: \tilde M \to M$
with $\varphi$. Then $\tilde \varphi$ is a smooth conformal map with the same
local behavior as $\varphi$.

We apply to $\tilde\varphi$ inequation \eqref{eq:qcinequality} from the proof of
the main quasiconformal inequality. 
As in the beginning of the proof of Theorem \ref{theo:conformal},
with $Q=1$ and $\beta\to 0$ we get for all $x\in \tilde M$ and all $\alpha>0$:
\[\tanh(\frac{r(\alpha)}2) \ge \sigma_0 \tanh(\frac\alpha2)\]
where $r(\alpha)$ is the radius of a ball in hyperbolic space
whose volume equals $|\tilde\varphi (B_{\tilde M}(x,\alpha))|$,
and $\sigma_0$ is the conformal dilation factor at $x$ (i.e.
$\sigma_0^n = |\jac \tilde\varphi(x)|$).

If we had $\sigma_0>1$, then for large enough $\alpha$ the above
inequality would yield $r(\alpha)\ge\infty$, a contradiction.
Therefore, $\sigma_0\le 1$ independently of $x$. Together with
the conformality of $\tilde\varphi$, this implies that
$\tilde\varphi$ is distance-nonincreasing.

If $\sigma_0=1$ (for one given $x$), then we have from the above inequality
$r(\alpha)\ge\alpha$ for all $\alpha$,
so that $\tilde\varphi$ maps balls of volume at most $V_{-1}(r)$ to balls
of volume at least $V_{-1}(r)$, while not increasing distances. This implies
that we have equalities in the Bishop and Günther inequalities, so that
$\tilde M$ and $\tilde N$ both have constant curvature $-1$ and
$\tilde\varphi$ is an isometry. 
\end{proof}

\begin{rema*}
\begin{enumerate}
\item The above result may seem weak in the sense that it asks for a conformal map,
which may not exist for given $M$ and $N$. However the hypothesis cannot
be weakened to quasiconformal as there are local $Q$-quasiconformal
diffeomorphisms of arbitrarily high supremum of the Jacobian. Using the
above method one can
only get bounds on averaged Jacobians, i.e. volume of balls.
\item Theorem \ref{theo:Ahlfors} can be interpreted as follows: given a
manifold $M$, if one can find in the same 
conformal class two complete metrics $g$ and $\sigma g$
such that $\Ric_g\ge -1$ and $K_{\sigma g}\le -1$, then
$\sigma$ is uniformly bounded above by $1$, and if there is a point
at which $\sigma(x)=1$ then $\sigma \equiv 1$. 
\end{enumerate}
\end{rema*}

\bibliographystyle{smfalpha}
\bibliography{biblio}

%%%%%%%%%%%%%%%%%%%%%%%%%%%%%%%%%%%%%%%%%%%%%%%%%%%%%%%%%%%%%%%
%%%%%%%%%%%%%%%%%%%%%%%%%%%%%%%%%%%%%%%%%%%%%%%%%%%%%%%%%%%%%%%
\end{document}